\documentclass[a4paper, 10pt, conference]{article}      

\usepackage[utf8]{inputenc}
\usepackage{times} 
\usepackage[T1]{fontenc}
\usepackage{graphicx}
\usepackage{amsmath}
\usepackage{amsthm}
\usepackage{amssymb}
\usepackage{mathrsfs}
\usepackage{hyperref}
\usepackage{color}
\usepackage{nomencl}

\newtheorem{remark}{Remark}
\newtheorem{lemma}{Lemma}
\newtheorem{definition}{Definition}

\newcommand{\eg}{{\it e.g. }}
\newcommand{\ie}{{\it i.e. }}
	
\newcommand{\bs}[1]{\boldsymbol{#1}}
\newcommand{\mc}[1]{\mathcal{#1}}
\newcommand{\mbb}[1]{\mathbb{#1}}
\newcommand{\mbf}[1]{\mathbf{#1}}
\newcommand{\diag}{\text{diag}}

\definecolor{fblue}{rgb}{0.1,0.3,0.7}

\title{\LARGE \bf Guaranteeing Input Tracking For  Constrained Systems: Theory and Application to Demand Response$^{*}$}

\author{Tomasz T. Gorecki$^{1,\dagger}$, Altu\u{g} Bitlislio\u{g}lu$^{1}$, Giorgos Stathopoulos$^{1}$ \\ and Colin N. Jones$^{1}$
\thanks{This work has received support from the Swiss National Science Foundation under the GEMS project (grant number 200021 137985) and the European Research Council under the European Union’s Seventh Framework Programme (FP/2007-2013)/ ERC Grant Agreement n. 307608 (BuildNet). }
\thanks{Automatic Control Laboratory, EPFL, Switzerland.}
\thanks{Corresponding author:  {\tt\small tomasz.gorecki@epfl.ch}}
}

\begin{document}
\maketitle
\thispagestyle{empty}
\pagestyle{empty}

\begin{abstract}
A method for certifying exact input trackability for constrained discrete time linear systems is introduced in this paper. A signal is assumed to be drawn from a reference set and the system must track this signal with a linear combination of its inputs. Using methods inspired from robust model predictive control, the proposed approach certifies the ability of a system to track any reference drawn from a polytopic set on a finite time horizon by solving a linear program. Optimization over a parameterization of the set of reference signals is discussed, and particular instances of parameterization of this set that result in a convex program are identified, allowing one to find the largest set of trackable signals of some class. Infinite horizon feasibility of the methods proposed is obtained through use of invariant sets, and an implicit description of such an invariant set is proposed. These results are tailored for the application of power consumption tracking for loads, where the operator of the load needs to certify in advance his ability to fulfill some requirement set by the network operator. An example of a building heating system illustrates the results. 
\end{abstract}

\section{Introduction}

This work proposes a methodology to handle input tracking for constrained discrete-time linear systems. More precisely, restriction of the inputs to a particular subspace of the input space is considered. For example, restricting the power consumption of a multi-input multi-output (MIMO) system to track a particular signal over time, the power consumption being a function of the inputs of this system, falls into our characterization of input tracking. In a nutshell, the method guarantees that, on a finite horizon, any signal drawn from a polyhedral set can be tracked exactly as a function of the inputs of our system, assuming only knowledge of this reference set. It ensures that whatever will be drawn from the reference set, the system can follow while still satisfying input and state constraints, as well as remain indefinitely feasible after that. The method also computes a controller that allows the system to track the reference causally when it is revealed only one step at a time. 

Another way to look at the method is as a disturbance rejection scheme, where the disturbance restricts the inputs to lie into some particular subspace. The control method proposed uses results from the robust MPC literature \cite{goulart_optimization_2006}.

Provision of tracking certifications for constrained systems is a problem that has recently been gaining attention. Contributions from the output tracking literature \cite{maeder_offset-free_2010, isidori_output_1990} study asymptotic tracking of the output of a reference generator for continuous-time systems in an unconstrained setting. Our problem differs from this classical setting in several regards: first, we consider discrete-time systems and exact tracking of the reference on a finite time horizon. Secondly, the references are not the output of a generator system, but are assumed to be drawn from a convex set. Finally, we consider tracking as a function of the inputs only for constrained systems. A recent work in which guarantees for inexact output tracking for constrained system have been considered is \cite{di_cairano_constrained_2013}, where the problem is tackled by means of robust invariance. To provide infinite horizon guarantees, so-called max-min invariant sets need to be computed. The conservatism in the computation of these sets leads to only inexact tracking guarantees and limits the scalability of the approach. We avoid these issues by ensuring tracking on a finite-time horizon, while still maintaining infinite horizon feasibility.

Related ideas have emerged in works trying to characterize the flexibility in power consumption of energy systems. Indeed, the increasing need for regulating power on the grid has pushed authors to consider load-side participation through programs such as Demand Response (DR) \cite{albadi_summary_2008}, which incentivize loads to modify their power consumption through price or request mechanisms, or to provide ancillary services \cite{galus_provision_2011}. \cite{m	amaasoumy_model_2014} considers power envelopes to characterize the power consumption flexibility of SISO systems, while the authors of \cite{sanandaji_improved_2013} consider the aggregated representation of Thermostatically Controlled Load (TCL) populations by simple battery models. These approaches are partly subsumed in this work. For example, the proposed method generalizes concepts of \cite{maasoumy_model_2014} by considering MIMO systems and a more general tracking characterization. \cite{vrettos_robust_2014} studies the aggregation of multiple buildings, and could be combined with our approach to leverage more flexibility from the buildings. Finally, \cite{oldewurtel_framework_2013} studies the DR potential of different types of loads regarding more practical aspects of the implementation and therefore complements this work.

This paper tackles the theoretical questions that arise from the constrained input tracking problem, and make the connection to the building application with an example. The contribution of the paper is three-fold: First, it gives a method for a priori certification of input trackability of a set of references. Secondly, it shows how the reference set can be optimized in a tractable fashion, for example to derive the largest possible set of references that can be tracked out of a particular class. Lastly, infinite horizon feasibility is discussed and an implicit characterization of an invariant set is proposed.

The paper is organized as follows: Section~\ref{sc:formulation} introduces the general formulation of the input trackability problem. Section~\ref{sc:solve} restricts the general problem to solve it with convex programming methods. Section~\ref{sc:optimDist} and ~\ref{sc:ScaleDist} introduce an approach to optimize over the reference set and presents a parameterization of the reference set that renders the problem convex. Section~\ref{sc:invariantSet} presents an implicit parametrization of the invariant set. Finally, Section~\ref{sc:Applications} illustrates the method on the case of a building heating ventilation and air conditioning (HVAC) system offering flexibility to a power grid operator.

\section{Formulation of the problem\label{sc:formulation}}

We consider constrained discrete-time linear system of the form:
\begin{equation}
\label{eq:dyn}
x_{i+1} = Ax_i + Bu_i\enspace,
\end{equation} 
where $x \in \mbb{R}^{n_x}$ and $u \in \mbb{R}^{n_u}$ are the state and input of the system, which are constrained to lie in the set:
\begin{equation*}
(x,u) \in \mbb X \times \mbb U = \mbb Z 
\end{equation*} 

The reference signal that the system is required to track at step $i$ is represented by $r_i \in \mbb R$. The reference $r$ is required to be tracked by a linear function of the control input $u$ so that:
\begin{equation}
g^Tu_i=r_i 
\label{eq:inputout}
\end{equation}
\begin{remark}
For simplicity, we consider tracking of a one-dimensional signal. Multi-dimensional signal tracking is a straightforward extension.
\end{remark}
\begin{remark}
In the situation at hand, the only source of uncertainty is the reference tracking signal $r$. It is assumed that future values of the reference are unknown but that the current value to be tracked is known. 
\end{remark}
\begin{remark}
In this paper, only input tracking is considered. The reason is two-fold: first, the application motivating this work (power consumption tracking for loads) only requires input tracking. Second, requiring exact tracking is generally not possible for the output tracking case. Even if the reference is known in advance, it might not be possible to track it exactly because of finite-time reachability issues. Formulations which relax the tracking requirement have to be used in this case \cite{di_cairano_constrained_2013}. 
\end{remark}

In the following, bold letters denote trajectories over a horizon, \eg ${\bf u}_i = (u_0^T,u_1^T,\ldots,u_{i-1}^T)^T$. For the remainder of the paper, we will consider a fixed horizon $N$, and so the subscript $N$ will be dropped when it is clear from the context. The map $\pi_i( x_i, r_i): \mbb X \times \mbb R \rightarrow \mbb R^{n_u}$ represents a causal feedback policy to be used at step $i$ after observing the current state $x_i$ and reference $r_i$. $\phi_i (x_0,\mbf r_i, \bs \pi_i)$ denotes the state of the system \eqref{eq:dyn} at step $i$, if it starts from state $x_0$, applies the control policy sequence $\bs{\pi}_i$, and receives the reference sequence $\mbf r_i$.

\begin{remark}
 It is assumed that the reference signal is observed at the time it needs to be tracked. Therefore the only source of uncertainty is the future values of the reference tracking signal $r$. However, the above definition and following results can be easily adapted to situations where the system has access to future reference values.
 \end{remark}
\begin{definition} \label{df:trackability}The set $\mc P \subset \mbb R^N$ is \textit{input trackable} over horizon $N$ by system \eqref{eq:dyn} in state $x_0$ if there exists a sequence of feedback policies $\bs \pi_\infty $ such that the following conditions hold:
\begin{equation}
\begin{aligned}
\phi_i (x_0,\mbf r_i, \bs \pi_i) \in \mbb X, \;  & \forall {\mbf r_{N-1}}  \in \mc P,\,\forall i \geq 0 \\
\pi_i(\mbf x_i , r_i) \in \mbb U, \;  & \forall {\mbf r_{N-1}}  \in \mc P,\,\forall i  \geq 0\\
g^T\pi_i(\mbf x_i, r_i) = r_i,\;& \forall {\mbf r_{N-1}}  \in \mc P,\,\forall i \in \{0,\ldots,N-1\}\\
\label{eq:trackcon}
\end{aligned}
\end{equation}
\end{definition}
The above definition asserts the ability of the system to track all possible reference sequences that can be drawn from the set $\mc P$ using causal feedback policies. It requires that the system tracks the reference up to time $N$ and then remains feasible. 

Given a set $\mc P$, our goal is to find a policy sequence $\bs \pi_\infty$ that satisfies the requirements of Definition \ref{df:trackability}. It is not tractable to look for a control policy over an infinite horizon. However, by enforcing the state $\phi_N$ to lie in a set $\mbb X_f$ which is controlled-invariant for the system \eqref{eq:dyn}, it is sufficient to find the first $N$ elements of the policy sequence, described by $ \bs \pi_{N-1}$, while inclusion of $\phi_N$ in $\mbb X_f$ guarantees the existence of the remaining elements of $ \bs \pi_\infty$. This method is a classical technique in Model Predictive Control. Even considering the finite sequence $\bs \pi_{N-1}$, it is difficult to solve this problem due to the infinite dimension of the policy space. The next section presents a finite parametrization of $\bs \pi_{N-1}$.

\section{Tracking with affine feedback policies}

\subsection{Tracking on a finite horizon\label{sc:solve}}
As reviewed in \cite{goulart_optimization_2006}, affine parameterizations have been introduced due to their nice computational properties (notably the convexity of the set of such admissible policies). Though additive state disturbances are not considered here, we show how to exploit the results in \cite{goulart_optimization_2006} in the following developments.

We consider the case where the set of feasible state and input constraints are polytopic and given by

\begin{equation}
\begin{aligned}
\label{eq:polycons}
&\mbb Z = \left\lbrace (x,u) \; \middle | \; Fx + Gu \leq c \right\rbrace  \enspace,\\
&\mbb X_f = \{ x \; | \; Hx \leq p\}\enspace.
\end{aligned}
\end{equation} 

The state sequence $\mbf{x}$ is fully determined by $\mbf{u}$ and the initial condition $x_0$, so that the dynamics and constraints are given by:
\begin{equation}
\begin{aligned}
&\mbf{x} = \mbf{A}x_0+ \mbf{Bu} \\
&\mbf{F}\mbf x+\mbf{Gu} \leq \mbf{c}\\
\label{eq:densedyn}
\end{aligned}
\end{equation}
in condensed form with appropriate matrices. Let us consider the following affinely parameterized control policy:

\begin{equation}
u_i=\sum_{j=0}^{i} M_{i,j}r_j + v_i, \quad \forall i \in \{0,...,N-1\}
\end{equation}

Using condensed notation:
\begin{equation}
\mbf{u} = \mbf{Mr+v}
\label{eq:ddistfeed}
\end{equation}
where $\mbf M$ is block lower triangular.

Policy~\eqref{eq:ddistfeed} is a reference sequence feedback policy. As established in \cite{goulart_optimization_2006}, it is at least as general as an affine state feedback policy, since past references and the current state are related in an affine fashion. Therefore the policy~\eqref{eq:ddistfeed} is compliant with the \textit{input trackability} conditions \eqref{eq:trackcon}. 

According to the input trackability conditions \eqref{eq:trackcon}, the set of admissible affine parameterizations $(\mbf{M},\mbf{v})$  can be written as:

\begin{equation}
\mathcal{F}(x_0, \mc P) = \left\{(\bf M, \bf v) \middle |
\begin{aligned} 
&\forall \mbf r \in \mc P\\
& \Sigma \mbf{u} \leq \sigma + \Xi x_0\\
&\mbf{u}=\mbf{M}\mbf{r}+\mbf{v}\\
& \Gamma \mbf{u} =\mbf r\\
&\mbf M \in \mc{LT}
\end{aligned}
\right\}\enspace,
\label{eq:dfset}
\end{equation}

where $\Sigma := \mbf{FB+G} $, $\sigma := \mbf{c}$ and $\Xi := -\mbf{FA}$ and $\Gamma := \mc I_N \otimes g^T$, with $\mc I_N$ the identity matrix of size $N$ and $\otimes$ the Kronecker product of the matrices. The structural constraint on $\mbf M$ such that it is a lower block triangular matrix is denoted as $\mbf M \in \mc{LT}$. In the sequel, we will simply write $\mc F(x_0)$ omitting the dependency on $\mc P$.

Let us also define the set of initial states $x$ for which an admissible policy exists:

\begin{equation}
\mc{X}:= \left\lbrace x \in \mathbb{R}^n \mid \mc F(x) \neq \emptyset \right\rbrace \enspace.
\end{equation}

\begin{lemma}
Both $\mc F (x)$ and $\mc X$ are convex.
\end{lemma}
\begin{proof}
\begin{equation}
\mc F(x) = \bigcap_{\mbf r \in \mc P}
\left\{(\bf M, \bf v) \middle |
\begin{aligned} 
& \Sigma \mbf{u} \leq \sigma + \Xi x_0\\
&\mbf{u}=\mbf{M}\mbf{r}+\mbf{v}\\
& \Gamma \mbf{u} =\mbf r\\
&\mbf M \in \mc{LT}
\end{aligned}
\right\}
\end{equation}
Written as such, $\mc F(x)$ is clearly the intersection of a family of convex sets and therefore is convex. $\mc{X}$ can be written as the projection on the $x$ subspace of the set defined as:
\begin{equation}
\left\{(\mbf M, \mbf v, x) \middle |
\begin{aligned} 
&\forall \mbf r \in \mc P\\
& \Sigma \mbf{u} \leq \sigma + \Xi x\\
&\mbf{u}=\mbf{M}\mbf{r}+\mbf{v}\\
& \Gamma \mbf{u} =\mbf r\\
&\mbf M \in \mc{LT}
\end{aligned}
\right\}
\end{equation}
This set can itself similarly be written as the intersection of a family of convex sets, and therefore is convex, as is its projection $\mc{X}_N$.
\end{proof} 
\begin{remark}
The reference set $\mc P$ needs not be time-invariant along the horizon, nor does it need to be time-uncorrelated. It does not even need to be convex for the previous lemma to hold.
\end{remark}

\begin{lemma}
\label{thm:dualization}
If $\mc P$ is a full-dimensional polyhedral set described by:
\begin{equation}
\label{eq:distSet}
\mc P = \left\{ \mbf r \;\middle|\; S\mbf r \leq h \right\}\enspace,
\end{equation}
then
\begin{equation}
\mathcal{F}(x_0) = \left\{(\bf M, \bf v) \middle |
\begin{aligned} 
&\exists \mbf{Z} \geq 0 \quad \text{s.t.}\\
&\Sigma \mbf{v} + \mbf{Z}^Th \leqslant \sigma + \Xi x_0 \\
&\mbf{Z}^T S= \Sigma\mbf{M}\\
& \Gamma\mbf M = \mc I_N, \, \Gamma\mbf v = 0\\
&\mbf M \in \mc{LT}
\end{aligned}
\right\}\enspace
\label{eq:dfset2dual}
\end{equation}
\end{lemma}
\begin{proof}
Notice first that since $\mc P$ is full-dimensional, the linear equalities $\forall \mbf r \in \mc P,\, \Gamma (\mbf M \mbf r + \mbf v) = \mbf r$ result in $\Gamma\mbf M = \mc I_N$ and $\Gamma\mbf v = 0$ by balancing both sides of the equation. 
For the inequality constraints, the universal quantifier can be removed via dualization. One can replace the universal quantifier with a maximization:
\begin{gather*}
\forall \mbf r \in \mc P,\, \Sigma (\mbf M \mbf r + \mbf v) \leq \sigma + \Xi x_0\\
\Leftrightarrow\\
\Sigma \mbf v + \max_{S\mbf r \leq h} \Sigma \mbf{Mr} \leq \sigma + \Xi x_0
\end{gather*}
where the maximization is taken row-wise. Dualizing these maximization problems \cite{goulart_optimization_2006, lofberg_automatic_2010} and introducing the dual variable $\mbf Z$ associated to the inequality constraints describing $\mc P$ in the different maximization problems, the description of the set $\mc F$ reduces to \eqref{eq:dfset2dual}.
\end{proof}

Restricting ourselves to polyhedral reference sets and affinely parameterized control policies, we can solve the\emph{ tracking certification problem} described by conditions \eqref{eq:trackcon} tractably by solving a single LP.

\subsection{Optimizing over the reference set\label{sc:optimDist}}

In general, one is looking for the "largest" set of reference signals that the system can track. Suppose the reference set $\mc P$ is parameterized with some parameters $\theta\in\mathbb{O} \subseteq \mathbb{R}^{n_\theta}$. Let us further assume that for all values of $\theta$, $\mc P$ is a polyhedral set.

For simplicity, let us redefine the notations of the previous sections as follows: For a particular value of $\theta$, we define:
\begin{align}
\mc{F}_{\theta}(x_0) &:= \left\{(\mbf M, \mbf v) \middle |
\begin{aligned} 
&\forall \mbf r \in \mc P(\theta)\\
& \Sigma (\mbf{Mr+v}) \leq \sigma + \Xi x_0\\
& \Gamma (\mbf{Mr+v}) =\mbf r\\
&\mbf M \in \mc{LT}
\end{aligned}
\right\} \enspace,
\end{align}
\begin{equation}
\mc{X}_{\theta}:= \left\lbrace x \in \mathbb{R}^n \mid \mc F_{\theta}(x) \neq \emptyset \right\rbrace \enspace,
\end{equation}
where $\mc{F}_{\theta}(x_0)$ is the set of all admissible affine disturbance feedback policies and $\mc{X}_{\theta}$ the set of feasible initial states for a particular value of $\theta$. We further define
\begin{equation}
\Theta(x_0) := \left\{ \theta \mid \mc{F}_{\theta}(x_0) \neq \emptyset \right\}
\end{equation}
as the set of all parameters defining the reference set for which there exists an admissible affine policy. 

Ultimately, the aim is to optimize over the "size" of the set $\mc P(\theta)$. Already, it is noticeable in characterization \eqref{eq:dfset2dual} that a linear parametrization of $S$ and $h$ in $\theta$ results in a problem with bilinear equalities and inequalities, which hints that the problem would in most cases be nonconvex. The following subsection presents instances of the problem for which it can be solved efficiently. Essentially we are looking for special cases where the parametrization of the disturbance set results in a convex search space $\Theta(x_0)$.

\subsection{Scaling of a fixed shape polytope\label{sc:ScaleDist}}

Let us consider the parametrization $\mc P(\theta) = \Lambda \mc T + \lambda = \left\{ \Lambda\mbf r+\lambda \mid \mbf r \in \mc T \right\}$ where $\mc T$ is a given polyhedron of dimension $N$, $\Lambda$ a diagonal scaling matrix of dimension $N\times N$ and $\lambda\in \mathbb{R}^N$ an offset vector. In the following, we show that we can efficiently optimize over $\theta = (\Lambda, \lambda)$. The following lemma is instrumental for this.

\begin{lemma}
\label{thm:scaling}
If
\begin{equation*}
\mc A = \left\{(\bf M, v) \middle |
\begin{aligned}
&\forall \mbf r \in \mc P\\
&\Sigma (\mbf{Mr+v}) \leq \sigma\\
&\Gamma\mathbf{v}=\lambda,\, \Gamma\mathbf{M} = \Lambda\\
&\mbf M \in \mc{LT}
\end{aligned}
\right\}
\end{equation*}
\begin{equation*}
\mc B = \left\{(\bf M, v) \middle |
\begin{aligned}
&\forall \mbf r \in \Lambda \mc P + \lambda\\
&\Sigma (\mbf{Mr+v}) \leq \sigma\\
&\Gamma\mathbf{v}=0,\, \Gamma\mathbf{M} = \mc I_N\\
&\mbf M \in \mc{LT}
\end{aligned}
\right\}
\end{equation*}
with $\Lambda$ a diagonal invertible matrix of appropriate dimension, then
\begin{align*}
\mc A = \emptyset \text{ if and only if }\mc B = \emptyset.
\end{align*}
\end{lemma}
\begin{proof}
Suppose $\mc A$ is not empty and $(\mbf{M, v})\in \mc A$. 
Then $(\mathbf{M} \Lambda^{-1}, \mathbf{v}-\mbf M\Lambda^{-1}\lambda) \in \mc B$. 
Indeed, because $\Lambda^{-1}$ is diagonal, we also have $\mathbf{M} \Lambda^{-1}\in \mc{LT}$. Moreover, $\forall \mbf r \in \Lambda \mc P + \lambda, \; \exists \tilde{\mbf r} \in \mc P: \mbf r = \Lambda \tilde{\mbf r} + \lambda$. 
Consequently, $\Sigma (\mbf M \Lambda^{-1}\mbf r+\mbf v - \mbf M\Lambda^{-1}\lambda) = \Sigma (\mbf M \Lambda^{-1}(\Lambda \tilde{\mbf r}+\lambda)+\mbf v- \mbf M\Lambda^{-1}\lambda) = \Sigma (\mbf M \tilde{\mbf r}+\mbf v) \leq \sigma$. The last inequality comes from the definition of $\mc A$.
Secondly, $\Gamma\mbf M\Lambda^{-1} = \Lambda\Lambda^{-1} = \mc I_N$ and $\Gamma(\mbf v-\mbf M\Lambda^{-1}\lambda)=\lambda-\Gamma\mbf M\Lambda^{-1}\lambda=0$. These together mean that $(\mbf M \Lambda^{-1}, \mbf v-\mbf M\Lambda^{-1}\lambda) \in \mc B$. Conversely, if $(\mbf M,\mbf v) \in \mc B$, then $(\mbf M\Lambda, \mbf v+\mbf M \Lambda^{-1}\lambda) \in \mc A$
\end{proof}

\begin{remark}
It can be useful to think of $\lambda$ as the nominal input trajectory of the system and the diagonal of $\Lambda$ as the flexibility around this nominal trajectory.
\end{remark}

Considering the parameterization $\mc P(\theta) = \diag(\theta_1)\mc T+\theta_2$, convexity of $\Theta_N(x_0)$ follows.

\begin{lemma}
\label{thm:cvxscale}
If $\theta =(\theta_1,\theta_2) \in \mathbb{R}_{+}^N\times \mathbb{R}^N$ where $\mathbb{R}_{+}$ is the real positive line and $\mc P(\theta) = \diag(\theta_1)\mc T + \theta_2$ where $\text{diag}(\theta_1)$ denotes the diagonal matrix with diagonal $\theta_1$, then $\Theta(x_0)$ is convex.
\end{lemma}
\begin{proof}

Following Lemma~\ref{thm:scaling} and removing the universal quantifier over $\mc P$ with dualization, we can write the description of $\Theta(x_0)$ as:

\begin{equation}
\label{eq:defTheta}
\Theta(x_0) =  \left\{\theta \;\middle |
\begin{aligned} 
&\exists (\mbf M, \mbf v, \mbf Z) \\
&\mbf Z \geq 0\\ 
&\Sigma \mbf{v} + \mbf{Z}^Th \leqslant \sigma+ \Xi x_0 \\
&\mbf{Z}^T S= \Sigma\mbf{M} \\
& \Gamma\mbf M = \diag (\theta_1), \, \Gamma\mbf v = \theta_2\\
&\mbf M \in \mc{LT}
\end{aligned}
\right\} 
\end{equation}

$\Theta(x_0)$ is the projection of a set defined by a family of linear equalities and inequalities, and therefore is convex.
\end{proof}

Note that we do not need to explicitly compute the projection to optimize over $\theta$. From a practical point of view, it means that we can optimize over all possible component-wise scaling of a polyhedral disturbance set efficiently. This includes, as a particular case, uniform scalings of polyhedron $\mc T$ if we consider matrices $\Lambda = \mu \mc I_N$. This allows us to find the largest volume reference set of given shape for a given horizon.

From Lemma~\ref{thm:scaling} and Lemma~\ref{thm:cvxscale} we can jointly optimize over the admissible control policies and reference sets in a computationally tractable way. This opens the possibility of re-optimizing the control policies after each step during closed loop operation. In such applications, availability of the initial admissible control policy that guarantees tracking and infinite horizon feasibility, ensures the recursive feasibility of the optimization problem, if the tracking requirement is not changed.

\subsection{An implicit terminal set \label{sc:invariantSet}}

We have shown how to find a control policy that will ensure \emph{input trackability} as specified by the conditions~\eqref{eq:trackcon}, and how to optimize over the reference set in a computationally tractable way, assuming that the terminal set $\mbb X_f$ is given. However, finding an explicit description of a controlled invariant set is usually difficult. In this section, we introduce an implicitly defined terminal condition, which ensures infinite horizon feasibility. The method scales well with dimension, as it does not require explicit set calculations.

From \eqref{eq:densedyn} and \eqref{eq:ddistfeed}, the terminal state $x_N$ is given by an affine function of the reference $\mbf r$

\begin{equation}
x_N=\bar{A} x_0 + \bar{B}\mbf M\mbf r +\bar{B}\mathbf{v} ,
\label{eq:xss}
\end{equation}
where $\bar{A} := \begin{bmatrix} 0 & \mc{I}_{n_x} \end{bmatrix} \mbf{A}$, $\bar{B} := \begin{bmatrix} 0 & \mc{I}_{n_x} \end{bmatrix} \mbf{B}$.

From the discussion in Section \ref{sc:formulation}, we require that $x_N$ lies in a controlled invariant set for all values of $\mbf r$. Note that this differs from the standard robust invariance condition since after the horizon of $N$ steps, there is no further reference to track and therefore no uncertainty. We follow the idea of \cite{limon_mpc_2008}, by enforcing that $x_N$ is a feasible steady state of the system for each value of the reference $\mbf r$:

\begin{equation}
\begin{aligned}
&x_{N}=Ax_{N}+Bu_{N}\\
&Fx_N+ Gu_N \leq c
\end{aligned}
\label{eq:conss}
\end{equation}

The input at the $N^{\text{th}}$ step $u_N$ is not specified by the control policy $\mbf{u = Mr + v}$, so we propose again an affine parametrization 
\begin{equation}
u_{N}=\mbf{M}_{ss}\mbf{r}+\mbf{v}_{ss}\enspace.
\label{eq:uss}
\end{equation} 

Combining~\eqref{eq:xss}-\eqref{eq:uss} gives the conditions
\begin{subequations} \label{eq:termconditions} 
 \begin{align}
& \mathbf{LM} + B\mbf M_{ss} = 0 \label{eq:term1} \\
& \mathbf{Lv} + B \mbf{v}_{ss} + \mathbf {T} x_0= 0\label{eq:term2}\\
&F \bar{B}[\mbf{Mr+v}] + G[\mbf{M}_{ss} \mbf{r} + \mbf{v}_{ss}  ] \leqslant c - F\bar{A} x_0 \label{eq:term3}
 \end{align}
 \end{subequations}
 where $\mathbf{L} := \left[A-I_{n_x}\right]\bar{B}$ and $\mbf T:= \left[A-I_{n_x}\right] \bar{A}  $.

The conditions \eqref{eq:termconditions} ensures that the control policy defined by $(\mbf{M, v}, \mbf{M}_{ss}, \mbf{v}_{ss})$ is able to drive the system to the set of admissible steady states, which is a control invariant set. Equations~\eqref{eq:term1}, \eqref{eq:term2} are linear equality constraints on $\mbf{M}, \mbf{M_{ss}},\mbf{v},\mbf{v_{ss}}$ and equations~\eqref{eq:term3} is an inequality identical in form to the one appearing in \eqref{eq:dfset}. Therefore, they can be handled following exactly the same principle as the one discussed in Lemma~\ref{thm:dualization} to recover tractable convex constraints. 

\remark{In some applications, such as building control, it is preferable to keep the system in a periodic steady state, due to the periodic nature of the disturbances and constraints. Periodic steady states can be easily incorporated into the definition of the set~\eqref{eq:termconditions} by representing the periodic system as a lifted version of the original system~\eqref{eq:dyn}, that describes the state evolution throughout a period of $N_P$ steps, and modifying the equations~\eqref{eq:xss}-\eqref{eq:uss}.

\begin{remark}
Addition of steady state control policy parametrization~$(\mbf{M}_{ss}, \mbf{v}_{ss})$ and conditions~\eqref{eq:termconditions} does not effect the results of Section~\ref{sc:ScaleDist} regarding the scaling of the reference set, since an admissible~$(\mbf{\tilde{M}}_{ss}, \mbf{\tilde{v}}_{ss})$ for the scaled reference set can be constructed as explained in the proof of Lemma~\ref{thm:scaling}.
\end{remark} 

Putting together equations~\eqref{eq:termconditions} and \eqref{eq:defTheta} leads us to solve the following convex problem:

\begin{equation}
\begin{array}{lll}
\text{minimize } & J \\
\text{subject to} &
\mbf Z,\mbf{Z_{ss}} \geq 0\\ 
&\Sigma \mbf{v} + \mbf{Z}^Th \leqslant \sigma+ \Xi x_0 \\
&\mbf{Z}^T S= \Sigma\mbf{M} \\ 
& \Gamma\mbf M = \diag (\theta_1), \, \Gamma\mbf v = \theta_2\\
&\mbf M \in \mc{LT} \\
& \mathbf{LM} + B\mbf M_{ss} = 0\\
& \mathbf{Lv} + B \mbf{v}_{ss} + \mathbf {T} x_0= 0\\
&F \bar{B}\mbf{v} + G\mbf{v}_{ss}  + \mbf{Z_{ss}}^Th \leqslant c - F\bar{A} x_0 \\
&\mbf{Z_{ss}}^TS = F \bar{B}\mbf{M} + G\mbf{M}_{ss}
\label{eq:finalprog}
\end{array}
\end{equation}

where the optimization variables are $\theta_1 \in \mbb{R}^N,\theta_2 \in \mbb{R}^N, \mbf{M} \in \mbb{R}^{Nn_u \times N}, \mbf{M_{ss}}\in \mbb{R}^{N_Pn_u \times N},\mbf v \in \mbb{R}^{Nn_u}, \mbf{v_{ss}} \in \mbb{R}^{N_Pn_u}, \mbf Z \in \mbb{R}^{n_h \times  n_c }$ and $\mbf{Z_{ss}} \in \mbb{R}^{n_h \times  {n_c}_P }$. $\mbf{Z_{ss}}$  is the matrix of Lagrange multipliers for equation~\eqref{eq:term3}, $n_h$ is the number of inequalities that represent the reference set $\mc{P}$, $n_c$ and ${n_c}_P$ are the dimensions of the inequalities described in~\eqref{eq:densedyn} and ~\eqref{eq:conss}}, respectively. The cost function $J$ needs to be convex in the optimization variables and is chosen according the problem at hand ({\it cf} Section~\ref{sc:Applications}). 

\section{Applications \label{sc:Applications}}

In this section, we present an application of the developed method, in which the flexibility in the power consumption of a building, around a nominal consumption profile, is characterized by means of a simple battery. This is highly desirable for assessing the capabilities of buildings to participate in demand response programs, in which participants are rewarded for their flexibility in consumption.

 We consider a simplified problem, where the external disturbances, such as weather, occupancy and solar radiation, are considered to be perfectly forecast and periodic. In a practical setting, the controller should also be robustified against uncertainty in the external disturbances, but we omit this issue in order to underline the presented methodology. Under the assumption of fully known disturbances and the linearity of the system, the building model can be described by~\eqref{eq:dyn}. We further assume that the thermal power consumption of the building is equal to the electricity consumption of the HVAC system. The building and the external disturbances represent a small office building in summer conditions, for which the model is obtained with the OpenBuild toolbox \cite{gorecki_openbuild_2014}. The building consists of three zones and the power input to each zone is considered as a separate input. The total power consumption of the HVAC system of the building is simply the sum of all inputs:
\begin{equation}
p_i = \mbf{1}^Tu_i
\end{equation}  
where $\mbf{1}$ is a vector of ones and $p \in \mbb{R}$ represents the total power consumption of the HVAC system. The HVAC system is required to satisfy the temperature constraints in each zone.

For the consumption flexibility description, we use the following \emph{battery} model, with constraints on the state-of-charge $s \in \mbb{R}$, and the rate of charge $r \in \mbb{R}$.

\begin{equation}
\begin{aligned}
 s_{i+1}  = &as_i + r_i\\
 0 \leqslant &~s_i \leqslant s_{max} \\
 -r_{max} \leqslant &~r_i \leqslant r_{max} \enspace .\\ 
\end{aligned}
\label{eq:dynbat}
\end{equation}

For a finite horizon length of $N$, the constrained model~\eqref{eq:dynbat} serves as a reference generator and describes the polytopic reference set $\mathcal{P}$, as the set of admissible power input trajectories $\bf r$ of the battery. We fix the shape of $\mathcal{P}$ by setting the battery parameters as;
\begin{equation}
a = 1,\;\; s_0 = s_{max}/2,\;\; s_{max}/r_{max}=5
\label{eq:fixshape}
\end{equation}

Given these constraints, the reference set $\mathcal{P}$ is fully parameterized by the maximum power limit $r_{max}$:

\begin{equation}
\begin{aligned}
\mathcal{P}(r_{max}):=\{ \mbf{r} \; | \; S \mbf{r} \leq h r_{max}\}
\end{aligned}
\label{eq:battery_shape}
\end{equation}

where $S\in \mbb{R}^{4N\times N}$ and $h\in\mbb{R}^{4N}$ can be constructed according to~\eqref{eq:dynbat} and \eqref{eq:fixshape}. Since $r_{max}$ is a scalar, $\mc{P}(r_{max})$ can be re-formulated as:

\begin{equation}
\mathcal{P}(r_{max} ,\bar{\mbf{p}}_{act}) = r_{max}\mathcal{T} + \bar{\mbf{p}}_{act}, \;\; \mathcal{T}:= \{S\mbf{r} \leqslant h \}
\end{equation}
where $\bar{\mbf{p}}_{act}$ represents the nominal power consumption trajectory of the HVAC system during the active period of demand response participation.

For the cost function of the building, we consider the case where the building is asked for symmetric up - down flexibility and only paid for the power limits. At the beginning of each day, the building finds the optimal battery it can support for a specific activation period, and the corresponding affinely parameterized control policy that guarantees \textit{input trackability} of all possible battery trajectories. The building then offers the corresponding power limits to the demand response operator, while minimizing its cost for the next day. The optimal reference set commitment problem to be solved by the building is given by:
\begin{equation}
\begin{array}{lll}
\text{minimize } & c_1^T\bar{\mbf{p}} - c_2^T \mbf{1} r_{max}\\
\text{subject to} & \text{\eqref{eq:finalprog}}\\
 & \theta_1 = \mbf{1}r_{max}\\
 & \theta_2 = \bar{\mbf{p}}_{act}
\end{array}
\end{equation}
where $\bar{\mbf{p}}$ represents the nominal power consumption over the whole horizon. The linear cost function represents the cost of electricity and flexibility reward paid to the building for the active period. The flexibility reward is taken to be twice the price of electricity. The extra cost that might arise from the tracking realizations are assumed to be compensated. The battery is only defined during the activation period, which covers 10 hours and is between 8:00 and 18:00. The reference signal $r$ is updated every hour, which is the discretization time step of the problem. The building uses predictions for the next day and also aims for a periodic steady state with a period of one day. Thus, the horizon length is $N=24$, with an additional $N_P=24$ steps for the periodic steady state.

  The control policy resulting from the optimization is simulated with randomly drawn samples from the reference set, as shown in Figure \ref{fig:open}. The flexibility captured by the optimally scaled battery ensures a storage capacity of $41.2\text{ kWh}$ with a power range of $16.5\text{ kW}$, which represents 36\% of the maximum nominal consumption of the HVAC system, standing at $45.4\text{ kW}$.

\begin{figure} [tbph]
\centering 
\includegraphics[width=\textwidth]{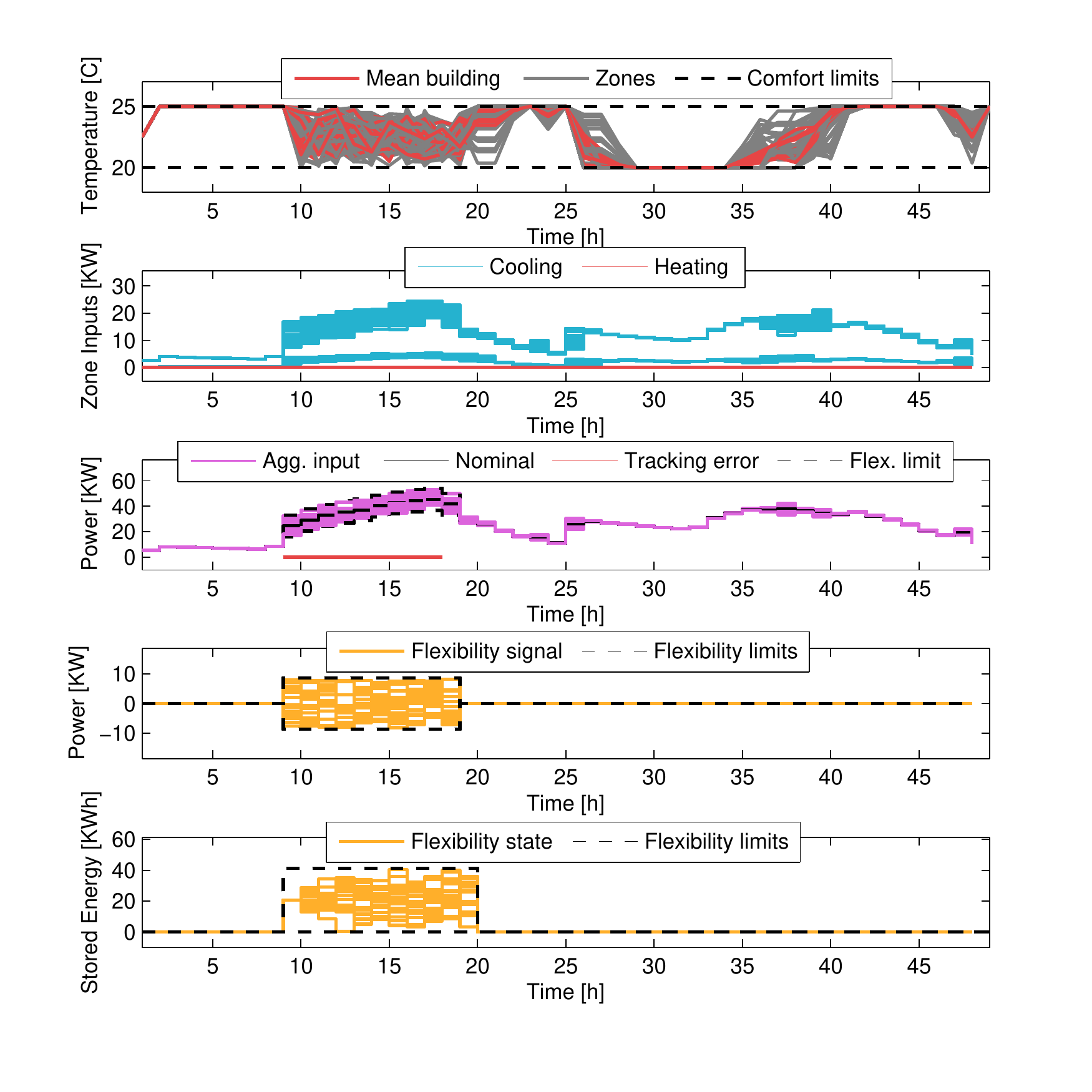}
\caption{Randomly sampled open loop trajectories, where the building uses affinely parameterized control policy found at the initial step. The top three figures represent the evolution of building outputs, inputs, and total power consumption of the HVAC system, while the last two figures show the flexibility signal $r$ and the state-of-charge $s$ described by the battery equations~\eqref{eq:dynbat}. Note that the second day of the simulation represents the periodic steady states computed by the control policy.}
\label{fig:open}
\end{figure}
%

\section{Conclusion}
We demonstrate in this paper how to certify in advance that a system can track (as a function of its inputs) references drawn from a reference signal set. The method utilizes a causal affine reference feedback policy to formulate a convex optimization problem that certifies trackability on a finite horizon window. Use of a terminal invariant set constraint also certifies that the system stays feasible indefinitely after the tracking period. The tracking reference set can be optimized tractably in some cases, allowing one to find the largest dimension-wise scaling of a set that can be tracked. An implicit characterization of a terminal set as the set of all feasible steady states is proposed for this particular setup. Results are illustrated by computing the power consumption flexibility of a building HVAC system. 

\bibliographystyle{IEEEtran}
\bibliography{IEEEabrv,biblio}

\end{document}